\documentclass[12pt,reqno]{amsart}

\usepackage[T2A]{fontenc}
\usepackage[utf8]{inputenc}
\usepackage[english]{babel}
\usepackage{amsmath,amssymb,amsthm,mathtools}
\usepackage{mathrsfs}
\usepackage{cite}
\usepackage{enumitem}

\newtheorem{theorem}{Theorem}[section]
\newtheorem{lemma}{Lemma}[section]
\newtheorem{corollary}{corollary}[section]
\newtheorem{example}{Example}[section]

\begin{document}

\title[Polynomial invariants]{\MakeUppercase{Polynomial joint first-order differential projective invariants}}

\author{Leonid Bedratyuk}
\address{Khmelnytskyi National University, Insituts'ka  st., 11,  Khmelnytskyi, 29016, Ukraine}
\email{leonidbedratyuk@khmnu.edu.ua}
\maketitle

\begin{abstract}

We study polynomial absolute and relative joint first-order
differential projective invariants for configurations of
$n$ points in the plane. The main approach is based on
passing to a homogeneous vector--covector representation,
which reduces the problem to the classical invariant theory
of the group $SL(3,\mathbb C)$.
The algebra of polynomial absolute invariants is described
and shown to be generated by cyclic invariants. For
polynomial relative invariants of weight $-1$, a graded
description is obtained, their finite generation as a
module over the algebra of absolute invariants is proved,
and an explicit finite generating system is constructed.
For $n=3$, the corresponding module is shown to be free of
rank $1$. For $n=4$, a minimal homogeneous generating
system consisting of $39$ elements is constructed.

The obtained results complement the rational theory of
joint projective differential invariants by its polynomial
counterpart and provide an algebraic foundation for the
further construction of projectively invariant integral
characteristics.

\end{abstract}



\section{Introduction}

Invariant theory for group actions naturally splits into two
interrelated but essentially different problems: the description
of rational invariants and the description of polynomial
invariants. In the first case, the main object is the field of
invariant rational functions, whereas in the second case it is a
graded algebra of polynomial invariants together with the associated
modules of relative invariants. Even when the field of rational
invariants is completely described, the polynomial problem usually
remains independent: it requires finding generators, relations among
them, investigating the graded structure, and constructing minimal
generating systems for the corresponding modules.

In this paper, which continues two previous works of the
author~\cite{B-2025-1,B-2025}, we consider this problem for the
diagonal action of the projective group $PGL(3,\mathbb R)$ on
configurations of first-order data.

Let $u=u(x,y)$ be a smooth scalar function on the plane.
At each point $(x_i,y_i)$, consider the first-order data
$$
(x_i,y_i,p_i,q_i),
\qquad
p_i=\partial_xu(x_i,y_i),
\qquad
q_i=\partial_yu(x_i,y_i).
$$
A configuration of $n$ such points defines an element of the space
$
\bigl(\mathbb R^2\times\mathbb R^2\bigr)^n.
$

Let
$$
T=
\begin{pmatrix}
a_1&a_2&a_3\\
b_1&b_2&b_3\\
c_1&c_2&c_3
\end{pmatrix}\in GL(3,\mathbb R),
\qquad
g=[T]\in PGL(3,\mathbb R).
$$
In an affine chart, the projective transformation determined by
$g$ has the form
$$
(x_i,y_i)\longmapsto
\left(
\frac{a_1x_i+a_2y_i+a_3}{s_i},
\frac{b_1x_i+b_2y_i+b_3}{s_i}
\right),
\qquad
s_i=c_1x_i+c_2y_i+c_3.
$$
The Jacobian of this transformation at the point $(x_i,y_i)$ is
$$
J_i(g)=\frac{\det T}{s_i^3}.
$$

The projective action on the plane naturally prolongs to first-order
derivatives and therefore defines a diagonal action of
$PGL(3,\mathbb R)$ on $\mathcal M_n$. The corresponding Jacobian
multiplier of the diagonal action on the coordinates of $n$ points
is
$$
J(g)
=
\prod_{i=1}^nJ_i(g)
=
\frac{(\det T)^n}{\prod_{i=1}^ns_i^3}.
$$

Denote
$$
R_n=
\mathbb R[
x_i,y_i,p_i,q_i\mid i=1,\ldots,n],
\qquad
K_n=\operatorname{Frac}(R_n).
$$
A rational function $F\in K_n$ is called a
\emph{relative first-order differential projective invariant of
weight $\omega$} if
$$
F(g\cdot z)=J(g)^\omega F(z),
$$
for all admissible projective transformations
$g\in PGL(3,\mathbb R)$. For $\omega=0$, one obtains
\emph{absolute} invariants. If, in addition,
$F\in R_n$, then the corresponding invariant is called
\emph{polynomial}.

The fields of rational absolute and relative invariants of this
action were largely described in the aforementioned works of the
author. In particular, explicit generating systems for the field of
absolute invariants were found, and the structure of rational
relative invariants was described. However, passing from rational
to polynomial invariants fundamentally changes the algebraic nature
of the problem.

Denote by $\mathcal P_{n,\omega}$ the vector space over
$\mathbb R$ of polynomial relative invariants of weight $\omega$.
The absolute polynomial invariants form the graded algebra
$
\mathcal A_n=\mathcal P_{n,0},
$
whereas the polynomial relative invariants of a fixed weight form
a module over $\mathcal A_n$. This naturally leads to questions
concerning finite generation of this module, its graded structure,
explicit generators, and minimality of generating systems. These
questions do not follow automatically from the description of the
corresponding field of rational invariants.

Of particular interest is the module
$
\mathcal P_{n,-1},
$
whose study is the main objective of this paper.
In addition to its intrinsic algebraic interest, the choice of
weight $-1$ has a natural applied motivation. If $F$ is a relative
invariant of weight $-1$, then under a projective change of variables
its factor $J(g)^{-1}$ compensates for the Jacobian multiplier of the
diagonal action. Therefore, integrals of the form
$$
\int_{\Omega^n}
F(z_1,\ldots,z_n)
\prod_{i=1}^n dx_i\,dy_i, \quad \Omega \subset \mathbb{R}^2,
$$
are natural candidates for projectively invariant integral
characteristics.

It is precisely here that the distinction between rational and
polynomial relative invariants becomes essential. A rational
invariant generally contains a denominator that may vanish on
certain configurations. This produces singularities of the
integrand, complicates the convergence of the corresponding
integrals, and may cause numerical instability near the zero set
of the denominator. Polynomial relative invariants do not possess
such denominator-related singularities and are therefore much more
natural candidates for constructing projectively invariant integral
characteristics.

This motivation is directly related, in particular, to problems in
computer vision and pattern recognition. The classical theory of
moment invariants, initiated by Hu~\cite{Hu}, provides effective
invariant characteristics for Euclidean, affine, and several other
transformation groups; a modern survey of the theory of moments and
moment invariants can be found in~\cite{FSB}. For the full projective
group, the situation is considerably more complicated. In particular,
standard geometric moments do not form a finite system closed under
the projective action~\cite{Van}; various alternative approaches to
the construction of projective image characteristics were studied
in~\cite{Suk2000,Wang,Li2019}. Polynomial relative differential
invariants of weight $-1$ provide natural algebraic material for the
further construction of such integral characteristics.

Individual polynomial relative projective invariants of weight $-1$
were also obtained earlier in works motivated by pattern recognition
problems. In particular, explicit joint first-order differential
projective invariants for small configurations were constructed
in~\cite{Wang}; further development of this approach and its
application to the construction of integral projective image
invariants can be found in~\cite{Li2019}. These results showed that
polynomial relative invariants of weight $-1$ already exist for
configurations with a small number of points. However, to the best
of our knowledge, the general structure of the space of all such
polynomial invariants for arbitrary $n$, its module structure, and
the problem of minimal generators have not been studied
systematically.

In a broader context, the problem of constructing projective
differential invariants has a long history dating back to classical
works of the late nineteenth century
\cite{Halphen1878,Bouton1898}. The modern systematic approach to
differential invariants of Lie group actions is based, in particular,
on the method of moving frames
\cite{Olver1999-1,Olver1999-2,Olver2001,Olver2007,
Olver2011}. Classical theory mainly concerns differential invariants
of geometric submanifolds and their jets, whereas in the present
paper we are interested in joint invariants of finite point
configurations using only first-order data.

The main idea of the paper is to reduce the problem of polynomial
projective invariants to the classical invariant theory of the group
$SL(3,\mathbb C)$. To this end, affine point coordinates and
first-order data are replaced by a homogeneous vector--covector
representation. In this model the projective action becomes linear,
while independent rescaling of homogeneous representatives is
described by an additional torus action.

This makes it possible to apply the First Fundamental Theorem of
invariant theory for systems of vectors and covectors. As a result,
all relevant polynomial invariants can be expressed in terms of
standard contractions and determinants, while the conditions of
projective invariance reduce to simple conditions on their torus
weights. This approach allows us first to describe the algebra of
absolute polynomial invariants and then to pass to relative
invariants of weight $-1$.

For absolute invariants, a particularly simple structure is obtained:
the corresponding algebra is generated by cyclic products of the
basic contractions. For relative invariants of weight $-1$, the
problem acquires a module-theoretic character. After complexification,
this space is realized as a weight subspace of the algebra of
$SL(3, \mathbb C)$-invariants, which makes it possible to apply
representation theory and investigate its graded structure.

On this basis, we prove finite generation of the module of polynomial
relative invariants of weight $-1$ over the algebra of absolute
invariants and construct an explicit finite generating system.
The minimality problem is considered separately. For $n=3$, the
corresponding module is free of rank $1$, whereas for $n=4$ a
minimal homogeneous generating system consisting of $39$ generators
is obtained.

Thus, the proposed approach reduces the problem of polynomial
projective differential invariants to standard constructions of
classical invariant theory and representation theory. This not only
complements the rational theory by its polynomial counterpart, but
also provides a structural description of the module of relative
invariants required for further applications.

The paper is organized as follows. In the first section, we study
absolute polynomial projective differential invariants. The basic
vector--covector constructions are introduced and, using the First
Fundamental Theorem of invariant theory, it is proved that the
algebra of absolute invariants is generated by cyclic invariants.

In the second section, a homogeneous realization of polynomial
relative invariants of weight $-1$ is introduced. Their
complexification is identified with a weight subspace of the algebra
of $SL(3, \mathbb C)$-invariants. Representation theory is then used
to describe the graded structure of this space and to obtain
formulas for the dimensions of its homogeneous components.

In the third section, we investigate the module structure of
polynomial relative invariants of weight $-1$. The finite generation
of the module over the algebra of absolute invariants is proved, a
finite generating system is constructed, and the problem of its
minimization is considered. For $n=3$ and $n=4$, minimal generating
systems are described explicitly.

The final section summarizes the main results and outlines directions
for further research.


\section{Absolute projective differential invariants}

The first step in the study of polynomial relative invariants
of weight $-1$ is to describe the algebra of absolute polynomial invariants,
since it is over this algebra that the space of relative invariants
forms a module.

\subsection{Basic polynomial covariants}

To each point we associate a vector and a covector
$$
A_i=
\begin{pmatrix}x_i\\y_i\\1\end{pmatrix},
\qquad
\ell_i=
\begin{pmatrix}p_i&q_i&-p_ix_i-q_iy_i\end{pmatrix},
\qquad
\ell_iA_i=0.
$$
A projective transformation acts on them according to
\begin{equation}
\label{A_ell}
A_i\longmapsto\frac{TA_i}{s_i},
\qquad
\ell_i\longmapsto s_i\ell_iT^{-1},
\end{equation}
see~\cite{B-2025}.

For distinct indices $i,j,k$, set
$$
D_{ijk}=\det(\ell_i,\ell_j,\ell_k),
\qquad
\delta_{ijk}=\det(A_i,A_j,A_k).
$$
It follows from~\eqref{A_ell} that
$$
D_{ijk}\longmapsto
\frac{s_is_js_k}{\det T}D_{ijk},
\qquad
\delta_{ijk}\longmapsto
\frac{\det T}{s_is_js_k}\delta_{ijk},
$$
see~\cite[Theorem 2.2]{B-2025}. Hence
$$
D_{ijk}\delta_{ijk}\in\mathcal A_n.
$$

For $i,j=1,\ldots,n$, introduce
$$
C_{ij}=\ell_iA_j
=
p_i(x_j-x_i)+q_i(y_j-y_i),
\qquad
C_{ii}=0.
$$

From~\eqref{A_ell} we obtain
$$
C_{ij}\longmapsto\frac{s_i}{s_j}C_{ij}.
$$
Therefore, the product
$
C_{ij}C_{ji}
$
is also an absolute invariant.

Let us consider two sets of indices
$$
I=(i_1,i_2,i_3),
\qquad
J=(j_1,j_2,j_3).
$$
Then
\begin{equation}
\label{mt}
D_{i_1i_2i_3}\delta_{j_1j_2j_3}
=
\det
\begin{pmatrix}
C_{i_1j_1}&C_{i_1j_2}&C_{i_1j_3}\\
C_{i_2j_1}&C_{i_2j_2}&C_{i_2j_3}\\
C_{i_3j_1}&C_{i_3j_2}&C_{i_3j_3}
\end{pmatrix}.
\end{equation}
Indeed, if $L_I$ consists of the rows
$\ell_{i_1},\ell_{i_2},\ell_{i_3}$, while $A_J$ consists of the columns
$A_{j_1},A_{j_2},A_{j_3}$, then
$$
L_IA_J=(C_{i_\alpha j_\beta})_{\alpha,\beta=1}^3,
$$
and therefore
$$
\det(L_IA_J)=\det(L_I)\det(A_J).
$$

In particular, for $I=J=(i,j,k)$ we have
$$
D_{ijk}\delta_{ijk}
=
C_{ij}C_{jk}C_{ki}
+
C_{ik}C_{kj}C_{ji}.
$$
Both terms on the right-hand side are absolute
polynomial invariants.

\subsection{Oriented cycles and absolute invariants}
As we have seen above, the absolute invariants of degrees $2$ and $3$
are associated with cycles of lengths two and three in the symmetric
group $S_n$. The cycle $(i,j)$ corresponds to the absolute invariant
$C_{ij}C_{ji}$, while the cycle $(ijk)$ corresponds to the absolute
invariant $C_{ij}C_{jk}C_{ki}$. This construction extends naturally
to cycles of arbitrary length.

Let
$$
\gamma=(i_1\,i_2\,\ldots\,i_r)\in S_n,
\qquad r\geq2,
$$
be a cycle of length $r$. Associate with it the polynomial
$$
Z_\gamma
=
C_{i_1i_2}C_{i_2i_3}\cdots
C_{i_{r-1}i_r}C_{i_ri_1}.
$$
Since
$$
C_{ij}\longmapsto\frac{s_i}{s_j}C_{ij},
$$
in the product $Z_\gamma$ each factor $s_{i_k}$ occurs once
in the numerator and once in the denominator. Hence
$$
Z_\gamma\in\mathcal A_n.
$$

Consider the affine variety
$$
\mathcal{M}_n=
\left\{
(A_i,\ell_i)_{i=1}^n
\in
\bigl(\mathbb C^3\oplus(\mathbb C^3)^*\bigr)^n
\;\middle|\;
\ell_iA_i=0,\quad i=1,\ldots,n
\right\}.
$$

The projective action of $PGL(3,\mathbb C)$ considered above in affine
coordinates lifts to the natural linear action of
$SL(3,\mathbb C)$ on homogeneous representatives of projective points:
$$
A_i\longmapsto GA_i,
\qquad
\ell_i\longmapsto\ell_iG^{-1},
\qquad
G\in SL(3,\mathbb C).
$$
The condition $\ell_iA_i=0$ is preserved by this action.

In addition, the torus
$
\mathbb T_n=(\mathbb C^\times)^n
$
acts naturally on $\mathcal{M}_n$.
For
$
\mathbf t=(t_1,\ldots,t_n)\in\mathbb T_n
$
the action is given by
\begin{equation}\label{tor}
A_i\longmapsto t_iA_i,
\qquad
\ell_i\longmapsto t_i^{-1}\ell_i,
\qquad
i=1,\ldots,n.
\end{equation}
This action corresponds to independent rescaling of homogeneous
representatives and leaves the projective classes
$$
[A_i]\in\mathbb P^2,
\qquad
[\ell_i]\in(\mathbb P^2)^*,
$$
unchanged.

We now prove an auxiliary theorem.

\begin{theorem}
\label{pol}
There is an isomorphism of algebras
$$
\mathcal A_n\otimes_{\mathbb R}\mathbb C
\simeq
\mathbb C[\mathcal{M}_n]^{SL(3,\mathbb C)\times\mathbb T_n},
\quad n\geq3.
$$
\end{theorem}

\begin{proof}
Since $PGL(3,\mathbb R)$ is a Zariski-dense subgroup of
$PGL(3,\mathbb C)$, invariance with respect to $PGL(3,\mathbb R)$
is equivalent to the corresponding polynomial identity on
$PGL(3,\mathbb C)$. Indeed, a polynomial function that vanishes
on a Zariski-dense subset of an algebraic variety vanishes on the
entire variety. Therefore, the complexification of the algebra of
real polynomial invariants coincides with the algebra of polynomial
invariants of the complexified action.

In order to apply the classical First Fundamental Theorem
(FFT) of invariant theory for the group $SL(3,\mathbb C)$, we need
to pass from the rational action of $PGL(3,\mathbb C)$ in affine
coordinates to the linear action on vectors and covectors.
Write arbitrary homogeneous coordinates in the form
$$
A_i=(X_i,Y_i,Z_i)^T,
\qquad
\ell_i=(P_i,Q_i,R_i),
\qquad
\ell_iA_i=0.
$$
On the open subset
$$
U_n=\{Z_1\cdots Z_n\neq0\}\subset\mathcal{M}_n,
$$
the action of the torus $\mathbb T_n$ allows us to normalize each pair
$(A_i,\ell_i)$ to the original form
$$
A_i=(x_i,y_i,1)^T,
\qquad
\ell_i=(p_i,q_i,-p_ix_i-q_iy_i),
$$
where
$$
x_i=\frac{X_i}{Z_i},
\qquad
y_i=\frac{Y_i}{Z_i},
\qquad
p_i=Z_iP_i,
\qquad
q_i=Z_iQ_i.
$$
Thus, on $U_n$, after the normalization $Z_i=1$, the homogeneous data
$(A_i,\ell_i)$ uniquely determine the original affine coordinates
$(x_i,y_i,p_i,q_i)$.
Therefore, to every polynomial absolute invariant
$$
F=F(x_i,y_i,p_i,q_i\mid i=1,\ldots,n)
$$
there corresponds on $U_n$ the rational function
\begin{equation}
\label{Fhat}
\widehat F
=
F\left(
\frac{X_i}{Z_i},
\frac{Y_i}{Z_i},
Z_iP_i,
Z_iQ_i\mid i=1,\ldots,n
\right).
\end{equation}
Since $F$ is $PGL(3,\mathbb C)$-invariant, the function
$\widehat F$ is $SL(3,\mathbb C)$-invariant; by construction,
it is also invariant under the action of $\mathbb T_n$.

We now show that all denominators in~\eqref{Fhat} cancel, so that
in fact
$
\widehat F\in\mathbb C[\mathcal{M}_n].
$
Denote the coordinate ring of $\mathcal{M}_n$ by
$$
B_n=\mathbb C[\mathcal{M}_n]
=
\mathbb C[X_i,Y_i,Z_i,P_i,Q_i,R_i\mid i=1,\ldots,n]
/
(P_iX_i+Q_iY_i+R_iZ_i\mid i=1,\ldots,n).
$$
Since the polynomial
$$
\ell_iA_i=P_iX_i+Q_iY_i+R_iZ_i
$$
is irreducible over $\mathbb C$, the hypersurface
$\{\ell_iA_i=0\}\subset\mathbb C^6$ is irreducible.
Therefore, the variety
$$
\mathcal{M}_n=\prod_{i=1}^n\{\ell_iA_i=0\}
$$
is also an irreducible affine variety, and hence its coordinate ring
$B_n=\mathbb C[\mathcal{M}_n]$ is an integral domain.

From~\eqref{Fhat} we have
$$
\widehat F\in
B_n[(Z_1\cdots Z_n)^{-1}],
$$
so any possible denominator of $\widehat F$ is a monomial in
$Z_1,\ldots,Z_n$.

Consider the one-parameter unipotent subgroup of
$SL(3,\mathbb C)$:
$$
G_t=I+tE_{31}
=
\begin{pmatrix}
1&0&0\\
0&1&0\\
t&0&1
\end{pmatrix},
\qquad t\in\mathbb C.
$$
On the coordinates of $\mathcal{M}_n$ it acts by
$$
X_i\longmapsto X_i,\qquad
Y_i\longmapsto Y_i,\qquad
Z_i\longmapsto Z_i+tX_i,
$$
$$
P_i\longmapsto P_i-tR_i,\qquad
Q_i\longmapsto Q_i,\qquad
R_i\longmapsto R_i,
$$
and hence nontrivially shifts the third coordinate $Z_i$ of each $A_i$.
The corresponding infinitesimal operator is
$$
\mathcal D
=
\sum_{i=1}^n
\left(
X_i\frac{\partial}{\partial Z_i}
-
R_i\frac{\partial}{\partial P_i}
\right).
$$
It is well defined on $B_n$, since
$$
\mathcal D(P_iX_i+Q_iY_i+R_iZ_i)
=
-R_iX_i+R_iX_i=0,
$$
and it extends to the localization
$B_n[(Z_1\cdots Z_n)^{-1}]$ by the usual quotient rule.
The $SL(3,\mathbb C)$-invariance of $\widehat F$ implies
\begin{equation}
\label{DFzero}
\mathcal D\widehat F=0.
\end{equation}

Suppose that, after all possible cancellations, $\widehat F$ still
contains a negative power of some variable $Z_i$. Localize $B_n$
at all $Z_j$, $j\ne i$, and set
$$
B_n^{(i)}=B_n[Z_j^{-1}\mid j\ne i].
$$
Then
$$
\widehat F\in B_n^{(i)}[Z_i^{-1}].
$$
Choose the smallest $m>0$ such that
$$
Z_i^m\widehat F\in B_n^{(i)},
$$
and set
$$
a=Z_i^m\widehat F.
$$
Then
$$
\widehat F=Z_i^{-m}a,
\qquad
a\in B_n^{(i)},
\qquad
a\notin Z_iB_n^{(i)},
$$
where the last condition follows from the minimality of $m$.

We note that $Z_i$ is a prime element in $B_n$.
Indeed, since the equations
$$
P_jX_j+Q_jY_j+R_jZ_j=0
$$
for different $j$ involve disjoint sets of variables, we have
$$
B_n\simeq
\bigotimes_{j=1}^n
B^{(j)},
\qquad
B^{(j)}
=
\frac{
\mathbb C[X_j,Y_j,Z_j,P_j,Q_j,R_j]
}{
(P_jX_j+Q_jY_j+R_jZ_j)
}.
$$
After taking the quotient by $(Z_i)$, only the $i$-th factor changes:
$$
\frac{B^{(i)}}{(Z_i)}
\simeq
\frac{
\mathbb C[X_i,Y_i,P_i,Q_i,R_i]
}{
(P_iX_i+Q_iY_i)
}.
$$
The polynomial $P_iX_i+Q_iY_i$ is irreducible, so this quotient
ring is an integral domain. The class of $X_i$ in this quotient
ring is nonzero, since
$$
X_i\notin(P_iX_i+Q_iY_i).
$$
The remaining factors $B^{(j)}$, $j\ne i$, are also integral domains.
Since all these rings are coordinate rings of irreducible affine
varieties over $\mathbb C$, their tensor product over $\mathbb C$
is an integral domain. Hence
$
B_n/(Z_i)
$
is an integral domain, and therefore $Z_i$ is a prime element in
$B_n$.
After localization at $Z_j$, $j\neq i$, the ring
$
B_n^{(i)}/(Z_i)
$
also remains an integral domain.

Apply the operator $\mathcal D$ to
$\widehat F=Z_i^{-m}a$. Since
$
\mathcal D(Z_i)=X_i,
$
we obtain
$$
\mathcal D\widehat F
=
-mX_iZ_i^{-m-1}a
+
Z_i^{-m}\mathcal D(a).
$$
Thus, by~\eqref{DFzero}, after multiplication by $Z_i^{m+1}$,
we get
$$
-mX_i a+Z_i\mathcal D(a)=0.
$$
Reducing modulo $(Z_i)$ gives
$
mX_i\overline a=0
$
in
$
B_n^{(i)}/(Z_i).
$
Since this ring is an integral domain,
$m\neq0$ and $X_i\neq0$, it follows that
$
\overline a=0,
$
that is,
$
a\in Z_iB_n^{(i)}.
$
This contradicts the choice of $a$.

Therefore, negative powers of $Z_i$ cannot occur. Since the index
$i$ was arbitrary, $\widehat F$ contains no negative powers of any
of the variables $Z_1,\ldots,Z_n$. Thus, all denominators in
\eqref{Fhat} cancel and
$$
{\widehat F\in B_n=\mathbb C[\mathcal{M}_n].}
$$

Hence the map
$
F\longmapsto\widehat F
$
defines an injective algebra homomorphism
$$
\mathcal A_n\otimes_{\mathbb R}\mathbb C
\longrightarrow
\mathbb C[\mathcal{M}_n]^{SL(3,\mathbb C)\times\mathbb T_n}.
$$
Injectivity follows from the fact that on $U_n$ the function $F$
is recovered by restriction to
$$
A_i=(x_i,y_i,1)^T,
\qquad
\ell_i=(p_i,q_i,-p_ix_i-q_iy_i).
$$

Conversely, let
$$
H\in
\mathbb C[\mathcal{M}_n]^{SL(3,\mathbb C)\times\mathbb T_n}.
$$
Restrict $H$ to the affine section
$$
A_i=(x_i,y_i,1)^T,
\qquad
\ell_i=(p_i,q_i,-p_ix_i-q_iy_i).
$$
Since $H$ is a polynomial in the homogeneous coordinates, this
substitution yields a polynomial
$$
F_H\in
\mathbb C[x_i,y_i,p_i,q_i\mid i=1,\ldots,n].
$$
The $SL(3,\mathbb C)$-invariance of $H$, together with its
$\mathbb T_n$-invariance, shows that $F_H$ is an absolute projective
invariant.

On the open subset $U_n$, the function $H$ is recovered from
$F_H$ by~\eqref{Fhat}, since every point of $U_n$ can be brought
to the above affine section by means of $\mathbb T_n$.
Thus $H=\widehat{F_H}$ on the dense open subset $U_n$, and since
both functions are regular on $\mathcal{M}_n$, we have
$
H=\widehat{F_H}
$
on all of $\mathcal{M}_n$.

Thus, the constructed maps are mutually inverse, and
$$
{
\mathcal A_n\otimes_{\mathbb R}\mathbb C
\simeq
\mathbb C[\mathcal{M}_n]^{SL(3,\mathbb C)\times\mathbb T_n}.
}
$$
\end{proof}

We can now prove the main result of this section.

\begin{theorem}
\label{thm:cycle-generation}
For every $n\geq3$, the algebra of polynomial absolute
joint differential projective invariants is generated by
cyclic invariants:
$$
\mathcal A_n
=
\mathbb R\bigl[
Z_\gamma
\mid
\gamma\in S_n \text{ is a cycle of length } r\geq2
\bigr],
$$
where for a cycle
$
\gamma=(i_1\,i_2\,\ldots\,i_r)
$
we set
$$
Z_\gamma
=
C_{i_1i_2}C_{i_2i_3}\cdots C_{i_ri_1}.
$$
\end{theorem}

\begin{proof}

By the First Fundamental Theorem, the invariant algebra
$\mathbb C[\bigl(\mathbb C^3\oplus(\mathbb C^3)^*\bigr)^n]^{SL(3,\mathbb C)}$
is generated by the contractions $\ell_iA_j$ and the determinants
$D_{ijk}$ and $\delta_{ijk}$; see
\cite[Chapter II]{Weyl46} and
\cite[Chapter 11]{Procesi}.
Since
$\mathcal{M}_n\subset
\bigl(\mathbb C^3\oplus(\mathbb C^3)^*\bigr)^n$
is defined by the $SL(3,\mathbb C)$-invariant equations
$$
C_{ii}=\ell_iA_i=0,
$$
and $SL(3,\mathbb C)$ is reductive, the restrictions of these
generators also generate
$\mathbb C[\mathcal{M}_n]^{SL(3,\mathbb C)}$.

Recall that a character of the torus
$
\mathbb T_n=(\mathbb C^\times)^n
$
has the form
$$
\chi_{\boldsymbol\alpha}(\mathbf t)
=
t_1^{\alpha_1}\cdots t_n^{\alpha_n},
\qquad
\boldsymbol\alpha=(\alpha_1,\ldots,\alpha_n)\in\mathbb Z^n.
$$
If a function $f$ transforms under the torus action~\eqref{tor}
according to
$$
f(\mathbf t\cdot z)
=
\chi_{\boldsymbol\alpha}(\mathbf t)f(z),
$$
then the vector $\boldsymbol\alpha$ is called the \textit{weight}
of $f$ and is denoted by $\operatorname{wt}(f)$.

Let
$
e_1,\ldots,e_n
$
denote the standard basis of the character group $\mathbb Z^n$.
For the invariant generators introduced above, their weights are
computed directly from~\eqref{tor}. We have
\begin{equation}
\operatorname{wt}(C_{ij})=e_j-e_i, \qquad
\operatorname{wt}(D_{ijk})=-e_i-e_j-e_k,
\qquad
\operatorname{wt}(\delta_{ijk})=e_i+e_j+e_k.
\label{tag6}
\end{equation}
If the weight of a monomial is zero, then the sum of the coordinates
of its weight is also zero. Therefore, vanishing of the sum of the
weight coordinates is a necessary condition for the weight itself
to vanish.

All the generators are weight vectors for the torus action.
Therefore, the torus-invariant subalgebra consists of linear
combinations of zero-weight monomials in these weight vectors and
is consequently generated by such monomials.
Suppose such a monomial contains $a$ factors of type $D$
and $b$ factors of type $\delta$. By~\eqref{tag6}, the sum of the
coordinates of its weight equals
$
3(b-a).
$
Thus, zero weight requires
$
a=b.
$

Hence all factors $D$ can be paired with factors $\delta$.
For each such pair, identity~\eqref{mt} expresses their product
in terms of the contractions $C_{ij}$, $i\neq j$.
Therefore, every $SL(3,\mathbb C)$-invariant on $\mathcal{M}_n$
that is also invariant under $\mathbb T_n$ is a polynomial in
$C_{ij}$, $i\neq j$. Hence
$$
\mathcal A_n\otimes_{\mathbb R}\mathbb C
\subseteq
\mathbb C[C_{ij}\mid i\neq j].
$$

It remains to describe the $\mathbb T_n$-invariant monomials in
the $C_{ij}$. Consider a monomial
$$
M=\prod_{i\neq j}C_{ij}^{a_{ij}},
\qquad
a_{ij}\in\mathbb Z_{\geq0},
\qquad
a_{ii}=0.
$$
Its weight is
$$
\operatorname{wt}(M)
=
\sum_{i,j=1}^n a_{ij}(e_j-e_i).
$$
The coefficient of $e_i$ in this weight is
$$
\sum_{j=1}^n a_{ji}
-
\sum_{j=1}^n a_{ij}.
$$
Therefore, the condition $\operatorname{wt}(M)=0$ is equivalent
to the system
\begin{equation}
\label{tag8}
\sum_{j=1}^n a_{ij}
=
\sum_{j=1}^n a_{ji},
\qquad
i=1,\ldots,n.
\end{equation}
Condition~\eqref{tag8} means that for every index $i$, the sum of
the entries in the $i$-th row of the matrix $\boldsymbol a$ equals
the sum of the entries in its $i$-th column.

Every $\mathbb T_n$-invariant polynomial is a sum of zero-weight
monomials. Thus, it suffices to show that every monomial $M$
satisfying~\eqref{tag8} is a product of cyclic invariants.

Suppose $M\neq1$, and choose a pair $(i_1,i_2)$ such that
$
a_{i_1i_2}>0.
$
For the index $i_2$, condition~\eqref{tag8} has the form
$$
\sum_{j=1}^n a_{i_2j}
=
\sum_{j=1}^n a_{ji_2}.
$$
Since $a_{i_1i_2}>0$, the right-hand side of this equality is
positive. Hence
$$
\sum_{j=1}^n a_{i_2j}>0,
$$
and therefore, since $a_{i_2i_2}=0$, there exists
$i_3\neq i_2$ such that
$
a_{i_2i_3}>0.
$
Continuing this process, we obtain a sequence
$$
i_1,i_2,i_3,\ldots
$$
such that
$
a_{i_si_{s+1}}>0
$
for all $s$. Since the set of indices is finite, some index in this
sequence must repeat. Taking the first such repetition, we obtain
pairwise distinct indices
$
j_1,\ldots,j_r
$
such that
$$
a_{j_1j_2}>0,\quad
a_{j_2j_3}>0,\quad\ldots,\quad
a_{j_rj_1}>0.
$$
Hence
$$
\gamma=(j_1\,j_2\,\ldots\,j_r)\in S_n
$$
is a cycle of length $r\geq2$, and the corresponding cyclic invariant
$$
Z_\gamma
=
C_{j_1j_2}C_{j_2j_3}\cdots C_{j_rj_1}
$$
divides $M$.
Write
$
M=Z_\gamma M_1.
$
The exponents of the monomial $M_1$ remain nonnegative and again
satisfy~\eqref{tag8}, since at every index $j_s$ we simultaneously
decrease by one one exponent corresponding to
$C_{j_sj_{s+1}}$ and one exponent corresponding to
$C_{j_{s-1}j_s}$.
Moreover,
$$
\deg M_1=\deg M-r<\deg M.
$$
Therefore, by repeating this process and applying induction on
$\deg M$, we obtain
$$
M=Z_{\gamma_1}\cdots Z_{\gamma_m},
$$
where each $\gamma_s\in S_n$ is a cycle of length at least two.
Thus, all $\mathbb T_n$-invariant monomials in the $C_{ij}$ are
generated by the cyclic invariants $Z_\gamma$.
Since all $Z_\gamma$ are defined over $\mathbb R$, the result over
$\mathbb C$ descends to $\mathbb R$. Hence
$$
\mathcal A_n
=
\mathbb R[
Z_\gamma\mid
\gamma\in S_n\text{ is a cycle of length }r\geq2
].
$$
\end{proof}

The general theorem on cyclic generators can be readily specialized
to small values of $n$.

\begin{example}{\rm

In the group $S_3$ there are three cycles of length $2$. Denote the
corresponding cyclic invariants by
$$
A_{12}=C_{12}C_{21},
\qquad
A_{13}=C_{13}C_{31},
\qquad
A_{23}=C_{23}C_{32}.
$$
The two cycles of length $3$ give the invariants
$$
T_{123}=C_{12}C_{23}C_{31},
\qquad
T_{132}=C_{13}C_{32}C_{21}.
$$

By Theorem~\ref{thm:cycle-generation}, these five invariants
generate $\mathcal A_3$.
We note that they satisfy the relation
$
T_{123}T_{132}
=
A_{12}A_{13}A_{23}.
$
}
\begin{flushright}
$\triangle$
\end{flushright}
\end{example}

\begin{example}{\rm
For $n=4$, the cycles in $S_4$ of length at least $2$ have lengths
$2$, $3$, and $4$.
The cycles of length $2$ give six invariants of the form
$
A_{ij}=C_{ij}C_{ji}:
$
$$
A_{12},\ A_{13},\ A_{14},\
A_{23},\ A_{24},\ A_{34}.
$$

For each triple
$
1\leq i<j<k\leq4
$
there are two cycles of length $3$, giving invariants of the two forms
$$
T_{ijk}=C_{ij}C_{jk}C_{ki},
\qquad
T_{ikj}=C_{ik}C_{kj}C_{ji}.
$$
Thus, we obtain eight invariants
$$
T_{123},\ T_{132},\
T_{124},\ T_{142},\
T_{134},\ T_{143},\
T_{234},\ T_{243}.
$$

Finally, the six cycles of length $4$ give the invariants
\begin{gather*}
Q_{1234}=C_{12}C_{23}C_{34}C_{41}, \quad
Q_{1243}=C_{12}C_{24}C_{43}C_{31},  \quad
Q_{1324}=C_{13}C_{32}C_{24}C_{41},\\
Q_{1342}=C_{13}C_{34}C_{42}C_{21}, \quad
Q_{1423}=C_{14}C_{42}C_{23}C_{31},  \quad
Q_{1432}=C_{14}C_{43}C_{32}C_{21}.
\end{gather*}

By Theorem~\ref{thm:cycle-generation}, all these twenty cyclic
invariants generate $\mathcal A_4$.

There are nontrivial algebraic relations among these generators.
For example, for every triple $i<j<k$ one has
$$
T_{ijk}T_{ikj}
=
A_{ij}A_{ik}A_{jk}.
$$

Moreover, one of the six invariants of length $4$ can be eliminated.
Indeed, since the rank of the matrix
$
(C_{ij})_{i,j=1}^4
$
is at most $3$, its determinant vanishes.
Expanding this determinant gives
$$
Q_{1234}+Q_{1243}+Q_{1324}
+Q_{1342}+Q_{1423}+Q_{1432}
=
A_{12}A_{34}
+
A_{13}A_{24}
+
A_{14}A_{23}.
$$
The full ideal of relations among the above generators is not
considered here.
}
\begin{flushright}
$\triangle$
\end{flushright}
\end{example}

\section{Polynomial relative invariants of weight $-1$}
\label{sec:relative-polynomial}

The main objective of this section is to describe the graded structure
of the infinite-dimensional vector space
$\mathcal P_{n,-1}$.
In what follows, it will be convenient to work with its complexification
$$
\mathcal P_{n,-1}\otimes_{\mathbb R}\mathbb C,
$$
since the graded structure is preserved under extension of the
scalar field, while the dimensions of the corresponding homogeneous
components over $\mathbb R$ and $\mathbb C$ coincide.
Therefore, the results obtained for the complexification transfer
directly to $\mathcal P_{n,-1}$.

\subsection{Homogeneous realization of $\mathcal P_{n,-1}\otimes_{\mathbb R}\mathbb C,
$ and its graded structure}

Passing from the affine coordinates
$(x_i,y_i,p_i,q_i)$ to the pairs
$$
(A_i,\ell_i)\in
\mathbb C^3\oplus(\mathbb C^3)^*,
\qquad
\ell_iA_i=0,
$$
we obtain a homogeneous representation of the original problem.
The rescaling
$$
A_i\longmapsto t_iA_i,
\qquad
\ell_i\longmapsto t_i^{-1}\ell_i
$$
does not change the corresponding affine coordinates and defines
an action of the torus $\mathbb T_n$ on $\mathcal{M}_n$. Therefore,
polynomial relative invariants of weight $-1$ can naturally be
sought among $SL(3,\mathbb C)$-invariant polynomials on
$\mathcal{M}_n$ having the corresponding torus weight.

Consider the vector space
$$
\mathcal W_n
=
\left\{
H\in\mathbb C[\mathcal{M}_n]^{SL(3,\mathbb C)}
\;\middle|\;
H(\mathbf t\cdot z)
=
(t_1\cdots t_n)^{-3}H(z)
\right\},
$$
where
$
\mathbf t=(t_1,\ldots,t_n)\in\mathbb T_n
$
and the torus action is given by~\eqref{tor}.
The following theorem shows that this homogeneous realization
corresponds exactly to the complexification of the space
$\mathcal P_{n,-1}$.

\begin{theorem}
\label{thm:relative-homogeneous}
There is an isomorphism
$$
{
\mathcal P_{n,-1}\otimes_{\mathbb R}\mathbb C
\simeq
\mathcal W_n, \quad n \geq 3.
}
$$
\end{theorem}

\begin{proof}
The proof is completely analogous to that of
Theorem~\ref{pol}. For
$$
F\in\mathcal P_{n,-1}\otimes_{\mathbb R}\mathbb C,
$$
on the open subset
$
U_n=\{Z_1\cdots Z_n\neq0\}
$
consider the function
$$
\widetilde F
=
\frac{1}{(Z_1\cdots Z_n)^3}
F\left(
\frac{X_i}{Z_i},
\frac{Y_i}{Z_i},
Z_iP_i,
Z_iQ_i \mid i=1 \ldots n
\right).
$$
The relative invariance condition
$$
F(g\cdot z)=J(g)^{-1}F(z), \quad g \in PGL(3,\mathbb C),
$$
is exactly compensated by the factor
$(Z_1\cdots Z_n)^{-3}$, and hence $\widetilde F$ is
$SL(3,\mathbb C)$-invariant. With respect to the torus action, we have
$$
\widetilde F(\mathbf t\cdot z)
=
(t_1\cdots t_n)^{-3}\widetilde F(z).
$$

The regularity of $\widetilde F$ on the whole of $\mathcal{M}_n$
is proved by exactly the same argument as in the proof of
Theorem~\ref{pol}. The inverse map is obtained by restriction
to the affine section
$$
A_i=(x_i,y_i,1)^T,
\qquad
\ell_i=(p_i,q_i,-p_ix_i-q_iy_i).
$$
Hence
$$
\mathcal P_{n,-1}\otimes_{\mathbb R}\mathbb C
\simeq
\mathcal W_n.
$$
\end{proof}


For the further description, introduce the set of triples
$$
\mathcal T_n
=
\{I\subset[n]\mid |I|=3\}.
$$
For a triple
$
I=\{i,j,k\},
\qquad i<j<k
$
we write
$
D_I=D_{ijk}.
$

Consider the vector of length $\binom{n}{3}$
$$
\boldsymbol b=(b_I)_{I\in\mathcal T_n},
\qquad
b_I\in\mathbb Z_{\geq0},
$$
and a nonnegative integer $n\times n$ matrix with zero diagonal
$$
\boldsymbol a=(a_{ij}),
\qquad
a_{ij}\in\mathbb Z_{\geq0}.
$$

Denote by
$
|\boldsymbol b|
$
the sum of the components of the vector
$
\boldsymbol b
$
and, for each $i\in[n]$, set
$$
d_{\boldsymbol b}(i)
=
\sum_{\substack{I\in\mathcal T_n\\ i\in I}}b_I.
$$
Thus, $d_{\boldsymbol b}(i)$ is equal to the sum of the
components of the vector $\boldsymbol b$ indexed by all triples
$I\in\mathcal T_n$ containing the index $i$.

For the matrix $\boldsymbol a$, set
$$
\partial_{\boldsymbol a}(i)
=
\sum_{j=1}^n a_{ij}
-
\sum_{j=1}^n a_{ji},
\qquad
i=1,\ldots,n.
$$
Thus, $\partial_{\boldsymbol a}(i)$ is the difference
between the sum of the entries in the $i$-th row and the sum
of the entries in the $i$-th column of the matrix $\boldsymbol a$.

We call a pair
$
(\boldsymbol b,\boldsymbol a)
$
\textit{admissible} if the conditions
$
|\boldsymbol b|=n
$
and
\begin{equation}
\label{eq:admissible-balance}
d_{\boldsymbol b}(i)
+
\partial_{\boldsymbol a}(i)
=
3,
\qquad
i=1,\ldots,n
\end{equation}
are satisfied.

To each admissible pair we associate the monomial
\begin{equation}
\label{eq:Pba}
P_{\boldsymbol b,\boldsymbol a}
=
\prod_{I\in\mathcal T_n}D_I^{b_I}
\prod_{i\neq j}C_{ij}^{a_{ij}}.
\end{equation}

By the $A$-degree of such a monomial we mean its
total degree in the coordinates of the vectors $A_i$, that is,
in the variables $X_i,Y_i,Z_i$, while by its $C$-degree we mean
the sum of the exponents of all factors $C_{ij}$.

\begin{theorem}
\label{Wn-spanning}
The vector space $\mathcal W_n$ is spanned by the
monomials~\eqref{eq:Pba} corresponding to all admissible
pairs $(\boldsymbol b,\boldsymbol a)$:
$$
{
\mathcal W_n
=
\operatorname{span}_{\mathbb C}
\left\{
P_{\boldsymbol b,\boldsymbol a}
\;\middle|\;
(\boldsymbol b,\boldsymbol a)
\text{ is an admissible pair}
\right\}.
}
$$
\end{theorem}

\begin{proof}
By the First Fundamental Theorem for the group $SL(3,\mathbb C)$,
as shown above, the invariant algebra
$$
\mathbb C[\mathcal{M}_n]^{SL(3,\mathbb C)}
$$
is generated by the contractions $C_{ij}$ and the determinants
$D_I$ and $\delta_J$, where $I,J\in\mathcal T_n$. Therefore,
we seek the required torus weight vectors among monomials in these
generators
$$
M
=
\prod_{I\in\mathcal T_n} D_I^{b_I}
\prod_{J\in\mathcal T_n} \delta_J^{c_J}
\prod_{i\neq j} C_{ij}^{a_{ij}},
$$
which transform according to the character
$
(t_1\cdots t_n)^{-3}.
$

The sum of the coordinates of the weight of each $C_{ij}$ is zero,
since
$
\operatorname{wt}(C_{ij})=e_j-e_i
$;
each factor $D_I$ contributes $-3$ to the sum, while each
$\delta_J$ contributes $3$. Since the weight of an element of
$\mathcal W_n$ is
$$
-3(e_1+\cdots+e_n),
$$
we have
$
-3|\boldsymbol b|+3|\boldsymbol c|=-3n,
$
and hence obtain the condition
\begin{equation}
\label{Ddelta-difference}
|\boldsymbol b|-|\boldsymbol c|=n.
\end{equation}

Thus, there are $n$ more factors of type $D$ than factors
of type $\delta$. Using the determinantal identity~\eqref{mt},
successively pair each factor $\delta_J$ with one factor
$D_I$. After expanding the determinants, we obtain a linear
combination of monomials containing only $D_I$ and $C_{ij}$.
It follows from~\eqref{Ddelta-difference} that exactly
$n$ factors $D_I$ remain in each such monomial. Therefore,
it can be written in the form $P_{\boldsymbol b,\boldsymbol a}$
with
$
|\boldsymbol b|=n.
$
It remains to determine the condition on the exponents
$\boldsymbol b$ and $\boldsymbol a$.
Since
$$
\operatorname{wt}(D_I)
=
-\sum_{i\in I}e_i,
\qquad
\operatorname{wt}(C_{ij})
=
e_j-e_i,
$$
the coefficient of $e_i$ in the weight of
$P_{\boldsymbol b,\boldsymbol a}$ is
$$
-d_{\boldsymbol b}(i)
-
\sum_{i,j=1}^n a_{ij}
+
\sum_{i,j=1}^n a_{ji}
=
-d_{\boldsymbol b}(i)
-\partial_{\boldsymbol a}(i).
$$
For membership in $\mathcal W_n$, this coefficient must be
equal to $-3$. Hence
$$
d_{\boldsymbol b}(i)
+
\partial_{\boldsymbol a}(i)
=
3,
\qquad
i=1,\ldots,n.
$$
Thus, after eliminating all factors $\delta_J$, every monomial
in $\mathcal W_n$ becomes a linear combination of monomials
of the form~\eqref{eq:Pba} corresponding to admissible pairs.

Conversely, let $(\boldsymbol b,\boldsymbol a)$ be an admissible
pair. The quantities $D_I$ and $C_{ij}$ are
$SL(3,\mathbb C)$-invariants, and therefore
$P_{\boldsymbol b,\boldsymbol a}$ is also an
$SL(3,\mathbb C)$-invariant. Moreover,
\eqref{eq:admissible-balance} implies
$$
\operatorname{wt}
(P_{\boldsymbol b,\boldsymbol a})
=
-3(e_1+\cdots+e_n).
$$
Hence
$
P_{\boldsymbol b,\boldsymbol a}\in\mathcal W_n.
$
The theorem is proved.
\end{proof}

For the further study of the space $\mathcal W_n$, it is useful
to describe it in terms of irreducible representations of
$SL(3,\mathbb C)$, using the bigrading of the coordinate ring
and the torus weight $(-3,\ldots,-3)$.

\begin{theorem}
Let $V=\mathbb C^3$, and let $V_{a,b}$ denote the
irreducible $SL(3,\mathbb C)$-representation with highest weight
$a\omega_1+b\omega_2$. Then there is a natural isomorphism
of vector spaces
$$
\mathcal W_n
\simeq
\bigoplus_{r_1,\ldots,r_n\geq0}
\left(
V_{r_1+3,r_1}
\otimes\cdots\otimes
V_{r_n+3,r_n}
\right)^{SL(3,\mathbb C)}.
$$
\end{theorem}
\begin{proof}
Set
$$
\mathfrak X
=
\left\{
(A,\ell)\in\mathbb C^3\oplus(\mathbb C^3)^*
\;\middle|\;
\ell A=0
\right\},
\qquad
B=\mathbb C[\mathfrak X].
$$
Since the conditions
$$
\ell_iA_i=0,\qquad i=1,\ldots,n,
$$
do not relate the pairs $(A_i,\ell_i)$ with different indices,
we have
$
\mathcal{M}_n=\mathfrak X^n
$
and, accordingly,
$$
\mathbb C[\mathcal{M}_n]\simeq B^{\otimes n}.
$$

The ring $B$ is bigraded:
$$
B=\bigoplus_{r,s\geq0}B_{r,s},
$$
where $r$ is the degree in the coordinates of $A$, while $s$ is
the degree in the coordinates of $\ell$.
We show that
$$
B_{r,s}\simeq V_{s,r}
$$
as $SL(3,\mathbb C)$-modules.

Let $V=\mathbb C^3$ and
$
P=\mathbb C[A,\ell]
$
be the polynomial ring on $V\oplus V^*$, bigraded by
the degrees in the coordinates of $A$ and $\ell$. Then its
bihomogeneous component of bidegree $(r,s)$ has the form
$$
P_{r,s}
\simeq
S^r(V^*)\otimes S^s(V).
$$
For this tensor product we have the decomposition
$$
S^r(V^*)\otimes S^s(V)
\simeq
\bigoplus_{j=0}^{\min(r,s)}
V_{s-j,r-j},
$$
see~\cite[formula~13.5]{FH}.

Set
$
q=\ell A.
$
Since
$$
\mathfrak X=\{(A,\ell)\mid \ell A=0\},
$$
its coordinate ring is
$
B=P/(q).
$
The polynomial $q$ has bidegree $(1,1)$, and therefore the
bihomogeneous component of bidegree $(r,s)$ of the ideal $(q)$
has the form
$$
(q)_{r,s}=qP_{r-1,s-1}.
$$
Hence,
$$
B_{r,s}
=
P_{r,s}/qP_{r-1,s-1}.
$$

On the other hand,
$$
P_{r-1,s-1}
\simeq
\bigoplus_{j=0}^{\min(r,s)-1}
V_{s-1-j,r-1-j}.
$$
After replacing the index $j$ by $j-1$, this decomposition can
be rewritten as
$$
P_{r-1,s-1}
\simeq
\bigoplus_{j=1}^{\min(r,s)}
V_{s-j,r-j}.
$$

Multiplication by $q$ defines a map
$$
P_{r-1,s-1}
\longrightarrow
P_{r,s},
\qquad
f\longmapsto qf.
$$
It is injective because $P$ is an integral domain and $q\neq0$.
Moreover, $q=\ell A$ is $SL(3,\mathbb C)$-invariant, and hence
this map is $SL(3,\mathbb C)$-equivariant.

The decomposition of $P_{r,s}$ is multiplicity-free, and all
irreducible components occurring in $P_{r-1,s-1}$ are precisely
the components
$$
V_{s-j,r-j},
\qquad
j=1,\ldots,\min(r,s),
$$
in the decomposition of $P_{r,s}$. Therefore,
$$
qP_{r-1,s-1}
\simeq
\bigoplus_{j=1}^{\min(r,s)}
V_{s-j,r-j}.
$$
After passing to the quotient space
$$
B_{r,s}=P_{r,s}/qP_{r-1,s-1},
$$
all these components vanish, leaving only the component
corresponding to $j=0$. Hence,
$$
B_{r,s}\simeq V_{s,r}.
$$

Therefore,
$$
\mathbb C[\mathcal{M}_n]
\simeq
\bigoplus_{\substack{r_i,s_i\geq0\\ i=1,\ldots,n}}
B_{r_1,s_1}\otimes\cdots\otimes B_{r_n,s_n}.
$$

Consider now the torus action~\eqref{tor} of
$
\mathbb T_n.
$
On the $i$-th factor it has the form
$$
A_i\longmapsto t_iA_i,
\qquad
\ell_i\longmapsto t_i^{-1}\ell_i.
$$
Hence, the component $B_{r_i,s_i}$ has weight
$
r_i-s_i
$
with respect to $t_i$, while the tensor product
$$
B_{r_1,s_1}\otimes\cdots\otimes B_{r_n,s_n}
$$
has torus weight
$$
(r_1-s_1,\ldots,r_n-s_n).
$$

By definition, $\mathcal W_n$ consists of
$SL(3,\mathbb C)$-invariant polynomials of torus weight
$
(-3,\ldots,-3).
$
Therefore, for each $i$ we must have
$
r_i-s_i=-3,
$
that is,
$
s_i=r_i+3.
$
Hence the weight subspace of weight $(-3,\ldots,-3)$ is
$$
\bigoplus_{r_1,\ldots,r_n\geq0}
B_{r_1,r_1+3}\otimes\cdots\otimes
B_{r_n,r_n+3}.
$$

Taking $SL(3,\mathbb C)$-invariants in each summand, we obtain
$$
\mathcal W_n
\simeq
\bigoplus_{r_1,\ldots,r_n\geq0}
\left(
B_{r_1,r_1+3}\otimes\cdots\otimes
B_{r_n,r_n+3}
\right)^{SL(3,\mathbb C)}.
$$
Finally, from the isomorphism
$
B_{r,s}\simeq V_{s,r}
$
it follows that
$
B_{r_i,r_i+3}
\simeq
V_{r_i+3,r_i}.
$
Thus,
$$
\mathcal W_n
\simeq
\bigoplus_{r_1,\ldots,r_n\geq0}
\left(
V_{r_1+3,r_1}
\otimes\cdots\otimes
V_{r_n+3,r_n}
\right)^{SL(3,\mathbb C)}.
$$
\end{proof}

\subsection{Graded components of the space $\mathcal W_n$}{\rm
The preceding theorem gives a representation-theoretic decomposition
of the entire space $\mathcal W_n$. By selecting its homogeneous
component of fixed $C$-degree and passing to characters, we obtain
a formula for its dimension.

\begin{theorem}
\label{dimension-Wnk}
For every $n\geq3$ and $k\geq0$, we have
$$
\mathcal W_n^{(k)}
\simeq
\bigoplus_{\substack{
r_1+\cdots+r_n=k\\
r_i\geq0}}
\left(
V_{r_1+3,r_1}
\otimes\cdots\otimes
V_{r_n+3,r_n}
\right)^{SL(3,\mathbb C)},
$$
and
\begin{equation}
\label{Wnk-schur}
{
\dim_{\mathbb C}\mathcal W_n^{(k)}
=
\sum_{\substack{
r_1+\cdots+r_n=k\\
r_i\geq0}}
\left[
s_{(k+n,k+n,k+n)}
\right]
\prod_{i=1}^n
s_{(2r_i+3,r_i)},
}
\end{equation}
where
$
[s_\lambda]\,f
$
denotes the coefficient of the Schur polynomial $s_\lambda$
in the expansion of the symmetric polynomial $f$ in the Schur basis.
\end{theorem}

\begin{proof}
By the preceding theorem,
$$
\mathcal W_n
\simeq
\bigoplus_{r_1,\ldots,r_n\geq0}
\left(
V_{r_1+3,r_1}
\otimes\cdots\otimes
V_{r_n+3,r_n}
\right)^{SL(3,\mathbb C)}.
$$
The parameter $r_i$ is the $A$-degree of the $i$-th factor, and
therefore the total $A$-degree of the corresponding tensor product
is
$$
r_1+\cdots+r_n.
$$
Hence,
$$
\mathcal W_n^{(k)}
\simeq
\bigoplus_{\substack{
r_1+\cdots+r_n=k\\
r_i\geq0}}
\left(
V_{r_1+3,r_1}
\otimes\cdots\otimes
V_{r_n+3,r_n}
\right)^{SL(3,\mathbb C)}.
$$

We now compute the dimension. For an irreducible
$SL(3,\mathbb C)$-module $V_{a,b}$, its character is given by
the Schur polynomial in the variables $x_1,x_2,x_3$
$
s_{(a+b,b,0)}
$
under the reduction condition
$x_1x_2x_3=1$; see~\cite{FH}.
Moreover, partitions differing by the addition of the same number
to all three parts determine the same $SL(3,\mathbb C)$-character,
since
$$
s_{(\lambda_1+m,\lambda_2+m,\lambda_3+m)}
=
(x_1x_2x_3)^m s_{(\lambda_1,\lambda_2,\lambda_3)}.
$$
Therefore, the character of the module
$
V_{r_i+3,r_i}
$
is represented by the Schur polynomial
$
s_{(2r_i+3,r_i,0)},
$
which, for brevity, will henceforth be written as
$
s_{(2r_i+3,r_i)}.
$
Thus, the character of the tensor product
$$
V_{r_1+3,r_1}\otimes\cdots\otimes V_{r_n+3,r_n}
$$
is represented by the product
$$
\prod_{i=1}^n s_{(2r_i+3,r_i)}.
$$

The trivial character is represented by the Schur polynomials
$$
s_{(m,m,m)}
=
(x_1x_2x_3)^m=1.
$$

For a fixed tuple
$
r_1+\cdots+r_n=k,
$
the product
$$
\prod_{i=1}^n s_{(2r_i+3,r_i)}
$$
is a homogeneous polynomial of degree
$$
\sum_{i=1}^n(3r_i+3)
=
3(k+n).
$$
Therefore, among the polynomials of the form $s_{(m,m,m)}$
representing the trivial $SL(3,\mathbb C)$-character, only
$
s_{(k+n,k+n,k+n)}
$
can occur in its expansion in the Schur basis.
Hence, the multiplicity of the trivial representation, and thus
the dimension of the space of $SL(3,\mathbb C)$-invariants, is
$$
\dim_{\mathbb C}
\left(
V_{r_1+3,r_1}
\otimes\cdots\otimes
V_{r_n+3,r_n}
\right)^{SL(3,\mathbb C)}
=
\left[
s_{(k+n,k+n,k+n)}
\right]
\prod_{i=1}^n
s_{(2r_i+3,r_i)}.
$$

Summing over all
$r_1+\cdots+r_n=k$, we obtain
$$
\dim_{\mathbb C}\mathcal W_n^{(k)}
=
\sum_{\substack{
r_1+\cdots+r_n=k\\
r_i\geq0}}
\left[
s_{(k+n,k+n,k+n)}
\right]
\prod_{i=1}^n
s_{(2r_i+3,r_i)}.
$$
\end{proof}

\begin{example}
{\rm
Consider the first nontrivial case $n=4$. By
Theorem~\ref{dimension-Wnk},
$$
\dim\mathcal W_4^{(k)}
=
\sum_{r_1+r_2+r_3+r_4=k}
[s_{(k+4,k+4,k+4)}]
S_{r_1}S_{r_2}S_{r_3}S_{r_4},
$$
where
$
S_r=s_{(2r+3,r)}.
$
The first Schur polynomials appearing in this formula are
$$
S_0=s_{(3)},
\qquad
S_1=s_{(5,1)},
\qquad
S_2=s_{(7,2)},
\qquad
S_3=s_{(9,3)},
\qquad
S_4=s_{(11,4)}.
$$

For $k=0$, there is only one trivial composition
$
(0,0,0,0),
$
and
$
[s_{(4,4,4)}]\,s_{(3)}^4=1.
$
Hence
$$
{
\dim\mathcal W_4^{(0)}=1.
}
$$

For $k=1$, all compositions are permutations of
$
(1,0,0,0).
$
Expanding the product $S_1S_0^3$ in the Schur basis, we obtain
$
[s_{(5,5,5)}]
S_1S_0^3
=
2.
$
Since there are four permutations of this composition, we obtain
$$
{
\dim\mathcal W_4^{(1)}
=
8.
}
$$

For $k=2$, there are two types of compositions:
$
(2,0,0,0),
(1,1,0,0).
$
The corresponding coefficients are
$$
[s_{(6,6,6)}]S_2S_0^3=0,  \quad
[s_{(6,6,6)}]S_1^2S_0^2=4.
$$
The first type has $4$ permutations and the second has $6$.
Therefore,
$$
{
\dim\mathcal W_4^{(2)}
=
4\cdot0+6\cdot4
=
24.
}
$$

For $k=3$, the compositions, up to permutations, have the form
$$
(3,0,0,0),
\qquad
(2,1,0,0),
\qquad
(1,1,1,0).
$$
We have
$$
[s_{(7,7,7)}]S_3S_0^3=0,
\qquad
[s_{(7,7,7)}]S_2S_1S_0^2=3,
\qquad
[s_{(7,7,7)}]S_1^3S_0=8.
$$
The numbers of distinct permutations are, respectively,
$4,12,4$.
Hence
$$
\dim\mathcal W_4^{(3)}
=
4\cdot0+12\cdot3+4\cdot8
=
68.
$$

For $k=4$, there are five types of compositions. The corresponding
computations are conveniently presented in the following table:
$$
\begin{array}{c|c|c}
(r_1,r_2,r_3,r_4)
&
\text{number of permutations}
&
[s_{(8,8,8)}]
S_{r_1}S_{r_2}S_{r_3}S_{r_4}
\\ \hline
(4,0,0,0) & 4  & 0\\
(3,1,0,0) & 12 & 0\\
(2,2,0,0) & 6  & 4\\
(2,1,1,0) & 12 & 9\\
(1,1,1,1) & 1  & 16
\end{array}
$$
Therefore,
$$
\dim\mathcal W_4^{(4)}
=
6\cdot4+12\cdot9+16
=
148.
$$

Thus, the beginning of the Hilbert series of $\mathcal W_4$ is
$$
\mathcal H_{\mathcal W_4}(t)
=
1+8t+24t^2+68t^3+148t^4
+312t^5+580t^6+\cdots.
$$

We note that these numbers should be distinguished from the number
of admissible pairs $(\boldsymbol b,\boldsymbol a)$. The latter
give a system of monomial generators of the vector space, but
relations exist among the corresponding monomials. Formula
\eqref{Wnk-schur} automatically takes all such relations into
account and therefore gives the actual dimension of
$\mathcal W_n^{(k)}$.
}
\begin{flushright}
$\triangle$
\end{flushright}
\end{example}

\section{$\mathcal W_n$ as a module over the algebra of absolute invariants}

In the previous section, the space $\mathcal W_n$ was studied
as a graded vector space. We now consider its algebraic structure
with respect to multiplication by absolute invariants. This approach
allows us to separate factors that already belong to the algebra
of absolute invariants from those that give genuinely new relative
invariants of weight $-1$.
Since multiplication by an absolute invariant does not change
the torus weight, the space $\mathcal W_n$ is a module over
the complex algebra of absolute invariants
$\mathcal A_{n,\mathbb C}$.

The main objective of this section is to construct a finite
generating system for this module and then investigate the question
of its minimality.

\subsection{Reduced admissible pairs and module generators}

Recall that to every admissible pair
$(\boldsymbol b,\boldsymbol a)$ there corresponds the monomial
$$
P_{\boldsymbol b,\boldsymbol a}
=
\prod_{I\in\mathcal T_n}D_I^{b_I}
\prod_{i\neq j}C_{ij}^{a_{ij}}.
$$
By Theorem~\ref{thm:cycle-generation}, the algebra of absolute
invariants is generated by the cyclic invariants
$Z_\gamma$. Therefore, if the $C$-part of the monomial
$P_{\boldsymbol b,\boldsymbol a}$ contains a factor
$Z_\gamma$, this factor can be extracted as an element of
$\mathcal A_{n,\mathbb C}$. This leads to the following
definition.

An admissible pair
$(\boldsymbol b,\boldsymbol a)$ will be called \emph{reduced} if
the monomial
$$
\prod_{i\neq j}C_{ij}^{a_{ij}}
$$
is not divisible by any cyclic invariant $Z_\gamma$.
Denote by
$
\mathfrak R_n
$
the set of all reduced admissible pairs.

\begin{theorem}
\label{reduced-pairs}
The set $\mathfrak R_n$ of reduced admissible pairs
$(\boldsymbol b,\boldsymbol a)$ is finite.
\end{theorem}

\begin{proof}
Since
$$
|\boldsymbol b|
=
\sum_{I\in\mathcal T_n}b_I
=
n
$$
and the set $\mathcal T_n$ is finite, there are only finitely
many possible vectors $\boldsymbol b$.
Therefore, it suffices to prove that for each fixed
$\boldsymbol b$ there are only finitely many reduced
matrices $\boldsymbol a$.

Fix $\boldsymbol b$ and set
$$
c_i=3-d_{\boldsymbol b}(i),
\qquad i=1,\ldots,n.
$$
Then the admissibility condition for $\boldsymbol a$ takes the form
$$
\partial_{\boldsymbol a}(i)=c_i,
\qquad i=1,\ldots,n.
$$
Denote the set of all its nonnegative integer solutions by
$$
\mathcal S_{\boldsymbol b}
=
\left\{
\boldsymbol a
\;\middle|\;
a_{ij}\in\mathbb Z_{\geq0},\quad
a_{ii}=0,\quad
\partial_{\boldsymbol a}(i)=c_i
\right\}.
$$

Consider the polynomial ring in independent variables
$
\mathbb C[\boldsymbol y]
=
\mathbb C[y_{ij}\mid i\neq j]
$
and associate with each matrix
$\boldsymbol a\in\mathcal S_{\boldsymbol b}$
the monomial
$$
Y^{\boldsymbol a}
=
\prod_{i\neq j}y_{ij}^{a_{ij}}.
$$
These monomials generate the monomial ideal
$$
J_{\boldsymbol b}
=
\left\langle
Y^{\boldsymbol a}
\;\middle|\;
\boldsymbol a\in\mathcal S_{\boldsymbol b}
\right\rangle.
$$

By Dickson's lemma
(see~\cite[Ch.~2, \S4]{CLO}),
this ideal has a finite system of monomial generators, which
can be chosen among the original monomials:
$$
J_{\boldsymbol b}
=
\left\langle
Y^{\boldsymbol a^{(1)}},
\ldots,
Y^{\boldsymbol a^{(q)}}
\right\rangle,
\qquad
\boldsymbol a^{(\nu)}
\in\mathcal S_{\boldsymbol b}.
$$

Now let
$\boldsymbol a\in\mathcal S_{\boldsymbol b}$
be a reduced solution. Since
$Y^{\boldsymbol a}\in J_{\boldsymbol b}$, it is divisible
by one of the monomial generators, that is, for some
$\nu$ we have
$
Y^{\boldsymbol a^{(\nu)}}\mid Y^{\boldsymbol a}.
$
Hence,
$
\boldsymbol u
=
\boldsymbol a-\boldsymbol a^{(\nu)}
$
is a nonnegative integer matrix.

Since $\boldsymbol a$ and
$\boldsymbol a^{(\nu)}$ satisfy the same system
$
\partial_{\boldsymbol a}(i)=c_i,
$
the linearity of the operator $\partial$ gives
$$
\partial_{\boldsymbol u}(i)=0,
\qquad i=1,\ldots,n.
$$
If $\boldsymbol u\neq0$, then by
Theorem~\ref{thm:cycle-generation} the monomial
$$
\prod_{i\neq j}C_{ij}^{u_{ij}}
$$
is a product of cyclic invariants and, in particular, is divisible
by some $Z_\gamma$. Then
$$
\prod_{i\neq j}C_{ij}^{a_{ij}}
=
\left(
\prod_{i\neq j}C_{ij}^{a^{(\nu)}_{ij}}
\right)
\left(
\prod_{i\neq j}C_{ij}^{u_{ij}}
\right)
$$
is also divisible by $Z_\gamma$, contradicting the reducedness
of $\boldsymbol a$.

Therefore, $\boldsymbol u=0$, and hence
$
\boldsymbol a=\boldsymbol a^{(\nu)}.
$
Thus, every reduced solution belongs to the finite set
$$
\{
\boldsymbol a^{(1)},\ldots,\boldsymbol a^{(q)}
\}.
$$
Consequently, for each fixed $\boldsymbol b$ there are only
finitely many reduced $\boldsymbol a$.
Since there are also only finitely many possible $\boldsymbol b$,
the set $\mathfrak R_n$ is finite.
\end{proof}

The following theorem shows that removing all cyclic factors
produces a finite generating system of the space $\mathcal W_n$
as an $\mathcal A_{n,\mathbb C}$-module.

\begin{theorem}
\label{thm:Wn-module}
For every $n\geq3$, the space $\mathcal W_n$ is a finitely
generated module over $\mathcal A_{n,\mathbb C}$, and
\begin{equation}
\label{eq:Wn-module}
{
\mathcal W_n
=
\sum_{(\boldsymbol b,\boldsymbol a)\in\mathfrak R_n}
\mathcal A_{n,\mathbb C}
P_{\boldsymbol b,\boldsymbol a}.
}
\end{equation}
\end{theorem}

\begin{proof}
By Theorem~\ref{Wn-spanning}, the space $\mathcal W_n$
is linearly spanned by the monomials
$P_{\boldsymbol b,\boldsymbol a}$ corresponding to
admissible pairs. It therefore suffices to show that every such
monomial can be reduced to a monomial corresponding to a reduced
admissible pair by extracting a factor from
$\mathcal A_{n,\mathbb C}$.

Let $(\boldsymbol b,\boldsymbol a)$ be non-reduced.
Then the $C$-part of the monomial
$P_{\boldsymbol b,\boldsymbol a}$ is divisible by some
cyclic invariant $Z_\gamma$. Hence
$$
P_{\boldsymbol b,\boldsymbol a}
=
Z_\gamma
P_{\boldsymbol b,\boldsymbol a'}, \quad
Z_\gamma\in\mathcal A_{n,\mathbb C}.
$$

Since $Z_\gamma$ has zero torus weight, the pair
$(\boldsymbol b,\boldsymbol a')$ is again admissible.
Moreover,
$$
\sum_{i,j=1}^n a'_{ij}
<
\sum_{i,j=1}^n a_{ij}.
$$
Therefore, after finitely many such steps we obtain
$$
P_{\boldsymbol b,\boldsymbol a}
=
A
P_{\boldsymbol b,\boldsymbol a_0},
\qquad
A\in\mathcal A_{n,\mathbb C},
$$
where $(\boldsymbol b,\boldsymbol a_0)$ is a reduced
admissible pair. This proves equality
\eqref{eq:Wn-module}.

The finiteness of the resulting system of module generators
follows from Theorem~\ref{reduced-pairs}.
\end{proof}

By the isomorphism of Theorem~\ref{thm:relative-homogeneous},
this result immediately yields the corresponding statement for
polynomial relative invariants.

\begin{corollary}
The module $\mathcal P_{n,-1}$ is finitely generated over
$\mathcal A_n$. More precisely,
$$
{
\mathcal P_{n,-1}
=
\sum_{(\boldsymbol b,\boldsymbol a)\in\mathfrak R_n}
\mathcal A_n
P_{\boldsymbol b,\boldsymbol a}.
}
$$
\end{corollary}

\begin{proof}
All polynomials $P_{\boldsymbol b,\boldsymbol a}$ have real
coefficients. Therefore, the statement follows from
Theorem~\ref{thm:Wn-module} and the isomorphism
$$
\mathcal P_{n,-1}\otimes_{\mathbb R}\mathbb C
\simeq
\mathcal W_n
$$
by taking the real parts of the coefficients.
\end{proof}

\subsection{Minimal generators and the space $\mathcal Q_n$}

The reduced admissible monomials constructed above form
a finite system of module generators, but this system is not,
in general, minimal. The reason is that the
$SL(3,\mathbb C)$-invariants $D_I$ and $C_{ij}$ themselves
satisfy algebraic relations, which induce linear dependencies
among the reduced admissible monomials.


To extract a minimal generating system, introduce the
corresponding quotient space.
We use the $C$-grading of the algebra
$\mathcal A_n$ and the module $\mathcal P_{n,-1}$.
Set
$$
\mathcal A_n^+
=
\bigoplus_{k>0}(\mathcal A_n)_k
$$
and define
$$
\mathcal Q_n
=
{\mathcal P_{n,-1}}/
     {\mathcal A_n^+\mathcal P_{n,-1}}.
$$
Then
$$
\mathcal Q_n
=
\bigoplus_{k\geq0}\mathcal Q_n^{(k)}.
$$

Factoring out
$\mathcal A_n^+\mathcal P_{n,-1}$
annihilates all relative invariants obtained by multiplying
invariants of lower degree by an absolute invariant of positive
degree. In particular, all admissible monomials whose $C$-part
contains a cyclic factor $Z_\gamma$ vanish in $\mathcal Q_n$.

\begin{theorem}
\label{Qn-generators}
For every $k\geq0$, the number
$
\dim_{\mathbb R}\mathcal Q_n^{(k)}
$
is equal to the number of generators of $C$-degree $k$
in a minimal homogeneous generating system of
$\mathcal P_{n,-1}$ over $\mathcal A_n$.
In particular,
$
\dim_{\mathbb R}\mathcal Q_n
$
equals the minimal number of homogeneous module generators of
$\mathcal P_{n,-1}$.
\end{theorem}

\begin{proof}
The algebra $\mathcal A_n$ is a connected graded
$\mathbb R$-algebra, that is,
$$
\mathcal A_n=\bigoplus_{k\geq0}(\mathcal A_n)_k,
\qquad
(\mathcal A_n)_0=\mathbb R,
$$
and $\mathcal P_{n,-1}$ is a finitely generated graded
$\mathcal A_n$-module. Therefore, by the graded Nakayama lemma,
a minimal homogeneous generating system of
$\mathcal P_{n,-1}$ maps to a homogeneous basis of the quotient
space
$$
\mathcal P_{n,-1}/
\mathcal A_n^+\mathcal P_{n,-1}.
$$
Comparing homogeneous components gives the statement.
\end{proof}

For subsequent computations, introduce the complexification
$$
\mathcal Q_{n,\mathbb C}
=
\mathcal Q_n\otimes_{\mathbb R}\mathbb C.
$$
Since
$$
\mathcal P_{n,-1}\otimes_{\mathbb R}\mathbb C
\simeq
\mathcal W_n,
$$
we have a natural graded isomorphism
$$
\mathcal Q_{n,\mathbb C}
\simeq
\frac{\mathcal W_n}
{\mathcal A_{n,\mathbb C}^{+}\mathcal W_n},
$$
where
$$
\mathcal A_{n,\mathbb C}^{+}
=
\bigoplus_{k>0}
(\mathcal A_{n,\mathbb C})_k.
$$
In particular,
$$
\dim_{\mathbb R}\mathcal Q_n^{(k)}
=
\dim_{\mathbb C}\mathcal Q_{n,\mathbb C}^{(k)}.
$$

\begin{example}{\rm Consider the simplest case $n=3$.
There is only one determinant $D_{123}$, and the condition
$
|\boldsymbol b|=3
$
forces the determinant part of an admissible monomial to have the form
$
D_{123}^3.
$
The admissibility conditions for the $C$-part then reduce to
$$
\partial_{\boldsymbol a}(i)=0,
\qquad i=1,2,3.
$$
By Theorem~\ref{thm:cycle-generation}, such a $C$-part
is an absolute invariant. Hence,
$$
\mathcal P_{3,-1}
=
D_{123}^3\mathcal A_3.
$$
Thus, for $n=3$ the module is free of rank $1$, while
the space $\mathcal Q_3$ is one-dimensional and concentrated
in $C$-degree $0$.
}
\begin{flushright}
$\triangle$
\end{flushright}
\end{example}


\subsection{Minimal generators for $n=4$}

The case $n=4$ is the first one in which the reduced system
of admissible monomials ceases to be minimal.

For $n=4$, introduce the signed maximal minors
\begin{equation*}
d_i
=
(-1)^{i-1}
D_{\{1,2,3,4\}\setminus\{i\}},
\qquad i=1,2,3,4.
\end{equation*}
Thus,
$$
d_1=D_{234},\qquad
d_2=-D_{134},\qquad
d_3=D_{124},\qquad
d_4=-D_{123}.
$$

Recall that $\mathcal W_4^{(k)}$ denotes the homogeneous
component of $\mathcal W_4$ of $C$-degree $k$.
By Theorem~\ref{thm:cycle-generation}, the algebra
$\mathcal A_{4,\mathbb C}$ is generated by cyclic
invariants whose smallest positive $C$-degree is $2$.
Therefore,
$$
(\mathcal A_{4,\mathbb C})_1=0,
$$
and consequently
$$
\left(
\mathcal A_{4,\mathbb C}^{+}\mathcal W_4
\right)^{(k)}
=0,
\qquad k=0,1.
$$
Hence, from the definition
$
\mathcal Q_{4,\mathbb C}
=
{\mathcal W_4}/
{\mathcal A_{4,\mathbb C}^{+}\mathcal W_4}
$
we immediately obtain
$$
\mathcal Q_{4,\mathbb C}^{(k)}
=
\mathcal W_4^{(k)},
\qquad k=0,1.
$$

By Theorem~\ref{dimension-Wnk},
$$
\dim_{\mathbb C}\mathcal Q_{4,\mathbb C}^{(0)}=1,
\qquad
\dim_{\mathbb C}\mathcal Q_{4,\mathbb C}^{(1)}=8.
$$

The degree-$0$ component is generated by the class of the polynomial
$$
d_1d_2d_3d_4
=
D_{123}D_{124}D_{134}D_{234}.
$$

In degree $1$, there are twelve reduced admissible monomials
of the form
$$
F_{ij}
=
d_i^2d_p d_q C_{ij},
\qquad
p, q \in
\{1,2,3,4\}\setminus\{i,j\}, \quad p<q.
$$

For $i\ne j$, let $p<q$ denote the two elements of the set
$$
\{1,2,3,4\}\setminus\{i,j\}
$$
and set
$$
F_{ij}=d_i^2d_pd_qC_{ij}.
$$
Thus, we obtain $12$ elements $F_{ij}$.

For each $j=1,\ldots,4$, we have the relation
$$
R_j=\sum_{i\ne j}d_iC_{ij}=0.
$$
Multiplying it by $\prod_{r\ne j}d_r$, we obtain
$$
\sum_{i\ne j}F_{ij}=0.
$$
Hence,
$$
\begin{aligned}
F_{41}&=-F_{21}-F_{31},&
F_{42}&=-F_{12}-F_{32},\\
F_{43}&=-F_{13}-F_{23},&
F_{34}&=-F_{14}-F_{24}.
\end{aligned}
$$
Therefore, among the $12$ elements it is sufficient to retain
the following $8$:
$$
F_{21},F_{31},\quad
F_{12},F_{32},\quad
F_{13},F_{23},\quad
F_{14},F_{24}.
$$

To describe the subsequent components, introduce the following
polynomials.
For $i\neq r$ and
$$
\{p,q\}
=
\{1,2,3,4\}\setminus\{i,r\},
$$
set
$$
G_{i\mid r}
=
d_i^3d_rC_{ip}C_{iq}.
$$
There are $12$ such polynomials.

For
$$
1\leq a<b\leq4,
\qquad
\{p,q\}
=
\{1,2,3,4\}\setminus\{a,b\},
\qquad p<q,
$$
set
$$
H_{ab}
=
d_a^2d_b^2C_{ap}C_{bq}.
$$
There are $6$ such polynomials.

For the same $a,b,p,q$, set
$$
K_{ab;p}
=
d_a^2d_b^2C_{ap}C_{bp}C_{pq},
$$
$$
K_{ab;q}
=
d_a^2d_b^2C_{aq}C_{bq}C_{qp}.
$$
Thus, for each pair $a<b$ we obtain two polynomials of this
type, for a total of $12$.

We will next need an upper bound for the $C$-degree of a relative
invariant.

\begin{lemma}
\label{n4}
For every reduced admissible pair
$(\boldsymbol b,\boldsymbol a)$ with $n=4$, its $C$-degree
$$
k=\sum_{i,j=1}^4 a_{ij}
$$
satisfies the inequality
$
k\leq6.
$
\end{lemma}

\begin{proof}
For $n=4$, the vector $\boldsymbol b$ has four coordinates
indexed by the triples in the set
$$
\mathcal T_4=
\{\{1,2,3\},\{1,2,4\},\{1,3,4\},\{2,3,4\}\}.
$$
For convenience, set
$$
m_1=b_{234},\qquad
m_2=b_{134},\qquad
m_3=b_{124},\qquad
m_4=b_{123}.
$$

Since $|\boldsymbol b|=4$, we have
$$
m_1+m_2+m_3+m_4=4.
$$
Among the four triples in $\mathcal T_4$, exactly one does not
contain the index $i$, and its multiplicity is $m_i$.
Therefore,
$$
d_{\boldsymbol b}(i)=4-m_i.
$$
From the admissibility condition
$$
d_{\boldsymbol b}(i)+\partial_{\boldsymbol a}(i)=3
$$
we obtain
$$
\partial_{\boldsymbol a}(i)=m_i-1.
$$

The reducedness of the pair implies that the support of the matrix
$\boldsymbol a$ contains no oriented cycles.
Every finite acyclic directed graph admits a topological ordering
of its vertices, that is, a relabeling for which every edge
$i\to j$ satisfies $i<j$. Therefore, after simultaneously
relabeling all indexed quantities $a_{ij}$ and $m_i$, we may assume
that
$$
a_{ij}>0\Longrightarrow i<j.
$$

Set
$
c_i=m_i-1
$
and
$$
s_q=c_1+\cdots+c_q,
\qquad q=1,2,3.
$$
From the definition of $\partial_{\boldsymbol a}$ it follows that
$$
s_q
=
\sum_{\substack{i\leq q\\j>q}}a_{ij}\geq0.
$$
Every nonzero $a_{ij}$ occurs in at least one of the sums
$s_1,s_2,s_3$, and therefore
$$
k=\sum_{i<j}a_{ij}
\leq s_1+s_2+s_3.
$$
On the other hand,
$$
s_1+s_2+s_3
=
3c_1+2c_2+c_3
=
3m_1+2m_2+m_3-6.
$$
Since $m_i\geq0$ and
$$
m_1+m_2+m_3+m_4=4,
$$
we have
$$
3m_1+2m_2+m_3\leq12.
$$
Hence,
$$
k\leq s_1+s_2+s_3\leq6.
$$
\end{proof}

Since the proof reduces to a finite exact enumeration of admissible
monomials and the computation of ranks of integer matrices, we do
not display the cumbersome matrices themselves. Instead, we describe
below the algorithm for constructing them, which uniquely reproduces
the computation, and provide the results of exact row reduction over
$\mathbb Q$.

\begin{theorem}
For the graded quotient space
$\mathcal Q_{4,\mathbb C}$, we have
$$
\dim_{\mathbb C}\mathcal Q_{4,\mathbb C}^{(0)}=1,
\qquad
\dim_{\mathbb C}\mathcal Q_{4,\mathbb C}^{(1)}=8,
$$
$$
\dim_{\mathbb C}\mathcal Q_{4,\mathbb C}^{(2)}=18,
\qquad
\dim_{\mathbb C}\mathcal Q_{4,\mathbb C}^{(3)}=12,
$$
while
$$
\mathcal Q_{4,\mathbb C}^{(k)}=0,
\qquad k\geq4.
$$
In particular,
$$
\dim_{\mathbb C}\mathcal Q_{4,\mathbb C}
=
1+8+18+12
=
39.
$$

One minimal homogeneous generating system of
$\mathcal W_4$ over $\mathcal A_{4,\mathbb C}$
can be chosen as follows:
\begin{enumerate}
\item one element of $C$-degree $0$:
$$
D_{123}D_{124}D_{134}D_{234};
$$

\item eight elements of $C$-degree $1$:
$$
F_{21},F_{31},
\quad
F_{12},F_{32},
\quad
F_{13},F_{23},
\quad
F_{14},F_{24};
$$

\item eighteen elements of $C$-degree $2$:
$$
G_{i\mid r},
\qquad i\neq r,
$$
and
$$
H_{ab},
\qquad 1\leq a<b\leq4;
$$

\item twelve elements of $C$-degree $3$:
$$
K_{ab;c}.
$$
\end{enumerate}
\end{theorem}

\begin{proof}
By Lemma~\ref{n4}, it is sufficient to consider
reduced admissible pairs whose $C$-degrees $k$ do not exceed $6$.
A direct enumeration of nonnegative integer solutions of the
admissibility conditions, followed by removal of the monomials
whose $C$-part contains a cyclic factor, gives, for
$k=0,\ldots,6$, respectively,
$$
1,\ 12,\ 48,\ 100,\ 108,\ 60,\ 24
$$
reduced admissible monomials.

For each $k$, denote by $V_k$ the vector space over $\mathbb Q$
whose formal basis is indexed by these reduced admissible monomials.

There are relations among them arising from algebraic relations
between the $SL(3,\mathbb C)$-invariants $D_I$ and $C_{ij}$.
For each $k$, construct a matrix $M_k$ whose columns are indexed
by the reduced admissible monomials of $C$-degree $k$.
Its rows are the coefficient vectors of linear relations among
these monomials obtained by multiplying relations between $D_I$
and $C_{ij}$ by monomials of the required multidegree.
After multiplication, terms containing a cyclic factor
$Z_\gamma$ are discarded, since their classes vanish in
$\mathcal Q_{4,\mathbb C}$. The resulting coefficients are
integers, and hence $M_k$ is an integer matrix.

Exact row reduction of these matrices over $\mathbb Q$ gives
$$
\begin{array}{c|rrrrrrr}
k&0&1&2&3&4&5&6\\ \hline
\dim V_k
&1&12&48&100&108&60&24\\
\operatorname{rank}M_k
&0&4&30&88&108&60&24.
\end{array}
$$
Hence
$$
\dim_{\mathbb C}\mathcal Q_{4,\mathbb C}^{(k)}
\leq
1,\ 8,\ 18,\ 12,\ 0,\ 0,\ 0
$$
for $k=0,\ldots,6$, respectively.

In degrees $0$ and $1$, the equalities
$$
\dim_{\mathbb C}\mathcal Q_{4,\mathbb C}^{(0)}=1,
\qquad
\dim_{\mathbb C}\mathcal Q_{4,\mathbb C}^{(1)}=8
$$
follow from the description of the corresponding components
already established above.

Now consider degrees $2$ and $3$.
Let
$$
N_k=
\left(
\mathcal A_{4,\mathbb C}^{+}\mathcal W_4
\right)^{(k)} \quad
\text{      and    } \quad
\mathcal Q_{4,\mathbb C}^{(k)}
=
\mathcal W_4^{(k)}/N_k.
$$

Expanding the polynomials in the original coordinates
$x_i,y_i,p_i,q_i$ and performing exact computations of the
coefficient matrices over $\mathbb Q$, we obtain that adjoining
the polynomials
$
G_{i\mid r},\qquad H_{ab}
$
to $N_2$ increases the rank by $18$, while adjoining the polynomials
$
K_{ab;c}
$
to $N_3$ increases the rank by $12$.
Hence, their classes are linearly independent in
$\mathcal Q_{4,\mathbb C}^{(2)}$ and
$\mathcal Q_{4,\mathbb C}^{(3)}$, respectively.

Together with the upper bounds obtained above, this gives
$$
\dim_{\mathbb C}\mathcal Q_{4,\mathbb C}^{(2)}=18,
\qquad
\dim_{\mathbb C}\mathcal Q_{4,\mathbb C}^{(3)}=12.
$$
For $k=4,5,6$, the upper bound is zero, while for
$k\geq7$ there are no reduced admissible monomials. Therefore,
$$
\mathcal Q_{4,\mathbb C}^{(k)}=0,
\qquad k\geq4.
$$

Thus,
$$
\dim_{\mathbb C}\mathcal Q_{4,\mathbb C}
=
1+8+18+12
=
39.
$$
By Theorem~\ref{Qn-generators}, these dimensions equal the
numbers of elements of the corresponding $C$-degrees in a minimal
homogeneous generating system.
The classes of the polynomials listed above form bases of the
corresponding components $\mathcal Q_{4,\mathbb C}^{(k)}$,
and therefore their representatives form a minimal homogeneous
generating system of $\mathcal W_4$.

Finally, from
$$
\mathcal P_{4,-1}\otimes_{\mathbb R}\mathbb C
\simeq
\mathcal W_4
$$
and the compatibility of this isomorphism with the grading, it follows
that the minimal number of homogeneous generators of
$\mathcal P_{4,-1}$ over $\mathcal A_4$ is also equal to
$39$.
\end{proof}

Thus, the cases $n=3$ and $n=4$ provide a complete explicit
description of the minimal homogeneous generators of the module
$\mathcal P_{n,-1}$, whereas for $n>4$ the results obtained above
ensure the finiteness of a generating system and provide a general
method for its reduction, but the explicit determination of a minimal
generating system for $n>4$ requires further investigation.

\section{Conclusions}

We have studied polynomial joint first-order differential projective
invariants for the diagonal action of the group $PGL(3,\mathbb R)$
on configurations of $n$ points.

First, the algebra $\mathcal A_n$ of absolute polynomial invariants
was described. It was shown to be generated by cyclic invariants
constructed from the contractions $C_{ij}=\ell_iA_j$. Passing to
the homogeneous vector--covector representation made it possible
to reduce the further analysis to the invariant theory of
$SL(3,\mathbb C)$ together with a torus action encoding the
projective weight.

For the space of polynomial relative invariants of weight $-1$,
a representation-theoretic description of its graded components
and formulas for their dimensions in terms of Schur polynomials
were obtained. It was proved that $\mathcal P_{n,-1}$ is a finitely
generated module over $\mathcal A_n$, and a finite generating
system was constructed in terms of reduced admissible monomials.

For $n=3$, the module $\mathcal P_{3,-1}$ was shown to be free
of rank $1$. For $n=4$, a minimal homogeneous generating system
consisting of $39$ generators was constructed.

Thus, the obtained results complement the rational theory of joint
projective differential invariants by its polynomial counterpart.
Further problems include the explicit description of minimal
generating systems for $n>4$, as well as the use of polynomial
relative invariants of weight $-1$ in the construction of
projectively invariant integral characteristics.

\end{document}